\documentclass[11pt]{amsart}
\usepackage{geometry}                % See geometry.pdf to learn the layout options. There are lots.
\usepackage{graphicx}
\usepackage{amssymb}
\usepackage{epstopdf}
\usepackage{tikz}
\usetikzlibrary{calc}
\usepackage{url}
\newtheorem{theorem}{\sc Theorem}
\newtheorem{lemma}{\sc Lemma}
\newtheorem{propn}{\sc Proposition}
\newtheorem{cor}{\sc Corollary}

\title{Quartic graphs with every edge in a triangle}
\author{Florian Pfender}
\address{Department of Mathematics and Statistics, University of Colorado Denver}

\author{Gordon F. Royle}
\address{Centre for the Mathematics of Symmetry and Computation \\ School of Mathematics and Statistics \\ University of Western Australia}
\begin{document}

\begin{abstract}
We characterise the  quartic (i.e. $4$-regular) multigraphs with the property that every edge lies in a triangle. The main result is that such graphs are either squares of cycles, line multigraphs of cubic multigraphs, or are obtained from these by a number of simple subgraph-replacement
operations. A corollary of this is that a simple quartic graph with every edge in a triangle is either the square of a cycle, the line graph of a 
cubic graph or a graph obtained from the line multigraph of a cubic multigraph by replacing triangles with copies of $K_{1,1,3}$.
\end{abstract}

\maketitle

\section{Introduction}

One of the most fundamental properties of a graph (see Diestel \cite{Diestel} for general background in graph theory) is whether it contains triangles, and if so, whether it has many triangles or few triangles, and many authors have studied classes of graphs that are extremal in some sense with respect to their triangles. While studying an unrelated 
graphical property, the authors were led to consider the class of regular graphs with the extremal property that {\em every edge} lies in a triangle; a property that henceforth we denote {\em the triangle property}. Clearly a disconnected graph has the triangle property if and only if each of its connected components does, and so it suffices to consider only connected graphs.

Although it seems impossible to characterise regular graphs of arbitrary degree with the triangle property, there are considerable structural restrictions on a graph with the triangle property when the degree is sufficiently low. When the degree is $2$ or $3$, the problem is trivial with $K_3$ and $K_4$ being the sole connected examples. When the degree is $4$ there are infinitely many graphs with the triangle property, but the restrictions are so strong that we can obtain a precise structural description of the class of 
connected quartic {\em multigraphs} with the triangle property.

To state the result, we first need two basic families of quartic multigraphs with the triangle property.
 The {\em squared $n$-cycle} $C_n^2$ is usually
defined to be the graph obtained from the cycle $C_n$ by adding an edge between each pair of vertices at distance $2$. However for our purposes,  we want to be more precise about multiple edges, and so for $n \geq 3$, we define $C_n^2$ as the Cayley {\em multigraph}\footnote{Thus the ``connection set'' of the Cayley graph is viewed as a multiset.} ${\rm Cay}(\mathbb{Z}_n, \{\pm 1, \pm 2\}).$ For $n=3$ and $n=4$, this creates graphs with multiple edges (see Figure~\ref{fig:squaredncyc}) but for $n \geq 5$ the graph is simple and either definition suffices. Inspection of Figure~\ref{fig:squaredncyc} makes it clear that in all cases the graph is a quartic multigraph with the triangle property.

\begin{figure}
\begin{center}
\begin{tikzpicture}[bend angle = 15]
\tikzstyle{vertex}=[circle, fill=gray, draw=black,inner sep=0.7mm]
\node [vertex] (w0) at (0:1cm) {};
\node [vertex] (w1) at (120:1cm) {};
\node [vertex] (w2) at (240:1cm) {};
\draw [thick,bend right] (w0) to (w1);
\draw [thick,bend left] (w0) to (w1);
\draw [thick,bend right] (w1) to (w2);
\draw [thick,bend left] (w1) to (w2);
\draw [thick,bend right] (w2) to (w0);
\draw [thick,bend left] (w2) to (w0);
\pgftransformxshift{4cm}
\node [vertex] (x0) at (45:1.2cm) {};
\node [vertex] (x1) at (135:1.2cm) {};
\node [vertex] (x2) at (225:1.2cm) {};
\node [vertex] (x3) at (315:1.2cm) {};
\draw [thick] (x0)--(x1)--(x2)--(x3)--(x0);
\draw [thick,bend right] (x0) to (x2);
\draw [thick,bend left] (x0) to (x2);
\draw [thick,bend right] (x1) to (x3);
\draw [thick,bend left] (x1) to (x3);
\pgftransformxshift{4.5cm}
\node [vertex] (v0) at (0:1.5cm) {};
\node [vertex] (v1) at (45:1.5cm) {};
\node [vertex] (v2) at (90:1.5cm) {};
\node [vertex] (v3) at (135:1.5cm) {};
\node [vertex] (v4) at (180:1.5cm) {};
\node [vertex] (v5) at (225:1.5cm) {};
\node [vertex] (v6) at (270:1.5cm) {};
\node [vertex] (v7) at (315:1.5cm) {};
\draw [thick] (v0)--(v1)--(v2)--(v3)--(v4)--(v5)--(v6)--(v7)--(v0);
\draw [thick] (v0)--(v2)--(v4)--(v6)--(v0);
\draw [thick] (v1)--(v3)--(v5)--(v7)--(v1);
\end{tikzpicture}
\end{center}
\caption{Squared cycles for $n=3$, $n=4$ and $n=8$.}
\label{fig:squaredncyc}
\end{figure}
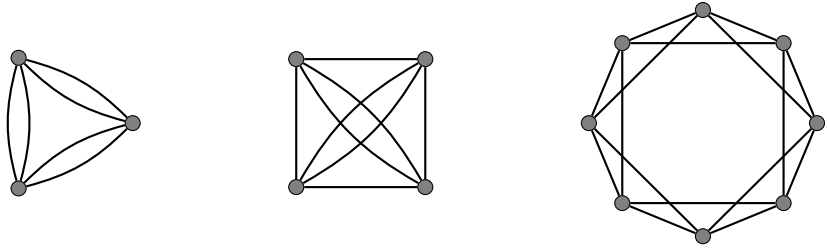

The second basic family is the family of {\em line multigraphs} of cubic multigraphs. For a multigraph $G$, we define the line multigraph $L(G)$ to have the edges of $G$ as its vertices, and where two edges of $G$ are connected by $k$ edges in $L(G)$ if they are mutually incident to $k$ vertices in $G$. In particular, if $e$ and $f$ are parallel edges in $G$, then there is a double edge in $L(G)$ between the vertices corresponding to $e$ and $f$; an example is shown in Figure~\ref{fig:linegraph}. If $G$ is a cubic multigraph, then $L(G)$ is a quartic multigraph. The edge set of $L(G)$ can be partitioned into cliques, with each vertex of degree $d$ in $G$ corresponding to a clique of size $d$ in $L(G)$. Therefore if $G$ is a {\em cubic} multigraph, the edge set of $L(G)$ can be partitioned into triangles, showing in a particularly strong way that $L(G)$ has the triangle property.

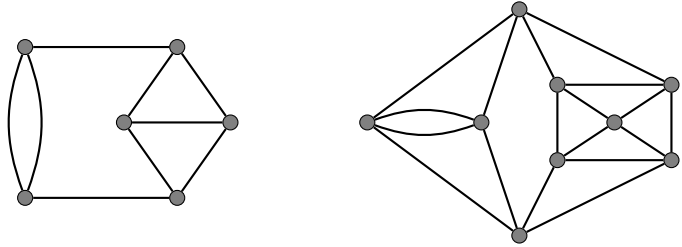
\begin{figure}
\begin{center}
\begin{tikzpicture}[bend angle = 20]
\tikzstyle{vertex}=[circle, fill=gray, draw=black,inner sep=0.7mm]
\node [vertex] (v0) at (0,0) {};
\node [vertex] (v1) at (0,2) {};
\node [vertex] (v2) at (2,2) {};
\node [vertex] (v3) at (2.7,1) {};
\node [vertex] (v4) at (2,0) {};
\node [vertex] (v5) at (1.3,1) {};
\draw [thick, bend right] (v0) to (v1);
\draw [thick, bend left] (v0) to (v1);
\draw [thick] (v0)--(v4)--(v5)--(v2)--(v1);
\draw [thick] (v3)--(v5);
\draw [thick] (v2)--(v3)--(v4);
\pgftransformxshift{4.5cm}
\pgftransformyshift{-0.5cm}
\node [vertex] (v0) at (0,1.5) {};
\node [vertex] (v1) at (1.5,1.5) {};
\node [vertex] (v2) at (2,0) {};
\node [vertex] (v3) at (2,3) {};
\node [vertex] (v4) at (2.5,1) {};
\node [vertex] (v5) at (2.5,2) {};
\node [vertex] (v6) at (3.25,1.5) {};
\node [vertex] (v7) at (4,1) {};
\node [vertex] (v8) at (4,2) {};
\draw [thick, bend right] (v0) to (v1);
\draw [thick, bend left] (v0) to (v1);
\draw[thick] (v0)--(v3);
\draw[thick] (v1)--(v3);
\draw [thick](v0)--(v2);
\draw [thick](v1)--(v2);
\draw [thick](v3)--(v5);
\draw [thick](v3)--(v8);
\draw [thick](v2)--(v4);
\draw [thick](v2)--(v7);
\draw [thick](v4)--(v5)--(v8)--(v7)--(v4);
\draw [thick](v4)--(v6)--(v8);
\draw [thick](v5)--(v6)--(v7);
\end{tikzpicture}
\end{center}
\caption{A cubic multigraph and its line multigraph}
\label{fig:linegraph}
\end{figure}

\begin{figure}
\begin{center}
\begin{tikzpicture}[bend angle = 15]
\tikzstyle{every node}=[circle, fill=gray, draw=black,inner sep=0.7mm];
\node (x) at (0,1) {};
\node (y) at (2,1) {};
\node (u) at (1,0) {};
\node (v1) at (0,-1) {};
\node (v2) at (2,-1) {};
\draw[style=thick] (v2)--(u)--(x)--(v1);
\draw[style=thick] (v2)--(y)--(u)--(v1);
\draw[style=thick] (v2) to [bend left] (v1);
\draw[style=thick] (v1) to [bend left] (v2);
\draw[style=thick, bend left] (x) to (y);
\draw[style=thick, bend left] (y) to (x);
\end{tikzpicture}
\end{center}
\caption{A $5$-vertex quartic multigraph with the triangle property}
\label{fig:5ex}
\end{figure}
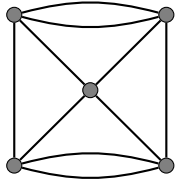

There are certain {\em subgraph-replacement} operations that can be performed on a graph while preserving the triangle property. A triangle $T$ (viewed as a set of three edges) is called {\em eligible} if it can be removed without destroying the triangle property, or if one of the three edges belongs to a triple edge. The first two operations apply to graphs with eligible triangles; in each an eligible triangle $T$ is removed, leaving the vertices and any other edges connecting them, and a new subgraph is attached in a specific way to the vertices of the removed triangle. 
Operation 1 replaces the removed edges with paths of length two, and joins the three new vertices with a new triangle, while Operation 2 replaces the triangle with a copy of $K_{1,1,3}$. Figure~\ref{fig:ops12} depicts these operations. 

\begin{figure}
\begin{center}
\begin{tikzpicture}[scale=0.8]
\tikzstyle{every node}=[circle, fill=gray, draw=black,inner sep=0.7mm];
\node[label=below left:$x$] (x) at (0,0) {};
\node[label=above left:$y$] (y) at (0,2) {};
\node[label=right:$z$] (z) at (2,1) {};
\draw[style=thick] (x)--(y)--(z)--(x);
\draw [->] (3,1)--(4,1);
\pgftransformxshift{5cm}
\node[label=below left:$x$] (x) at (0,0) {};
\node[label=above left:$y$] (y) at (0,2) {};
\node[label=right:$z$] (z) at (2,1) {};
\node (u) at (0,1) {};
\node (v) at (1,0.5) {};
\node (w) at (1,1.5) {};
\draw[style=thick] (x)--(u)--(y)--(w)--(z)--(v)--(x);
\draw[style=thick] (u)--(v)--(w)--(u);

\pgftransformxshift{-5cm}
\pgftransformyshift{-4cm}
\node[label=below left:$x$] (x) at (0,0) {};
\node[label=above left:$y$] (y) at (0,2) {};
\node[label=right:$z$] (z) at (2,1) {};
\draw[style=thick] (x)--(y)--(z)--(x);
\draw [->] (3,1)--(4,1);
\pgftransformxshift{5cm}
\node[label=below left:$x$] (x) at (0,0) {};
\node[label=above left:$y$] (y) at (0,2) {};
\node[label=right:$z$] (z) at (2,1) {};
\node (u) at (1,1.4) {};
\node (v) at (1,0.6) {};
\draw[style=thick] (v)--(x)--(u)--(y)--(v)--(z)--(u)--(v);
\end{tikzpicture}
\end{center}
\caption{Operations $1$ and $2$ where $\{xy, yz, zx\}$ is an eligible triangle}
\label{fig:ops12}
\end{figure}
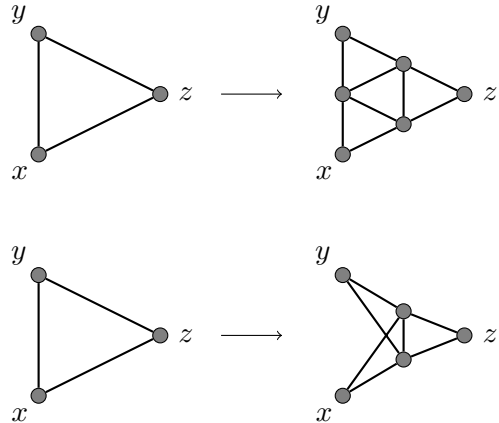

The third and fourth operations replace specific subgraphs with larger subgraphs, and are best described by pictures --- see Figure~\ref{fig:ops34} --- rather than in words. The named vertices are the points of attachment of the subgraph to the remainder of the graph and remain unchanged. In Operation 3, the left-hand subgraph is necessarily an induced subgraph, but in Operation 4 it is possible that $x$ and $y$ are connected by a double edge, in which case the left-hand subgraph is the squared $4$-cycle and the right-hand graph is the graph shown in Figure~\ref{fig:5ex}. It is not possible for $x$ and $y$ to be connected by a single edge (see Lemma~\ref{twodegthree} below). 
 
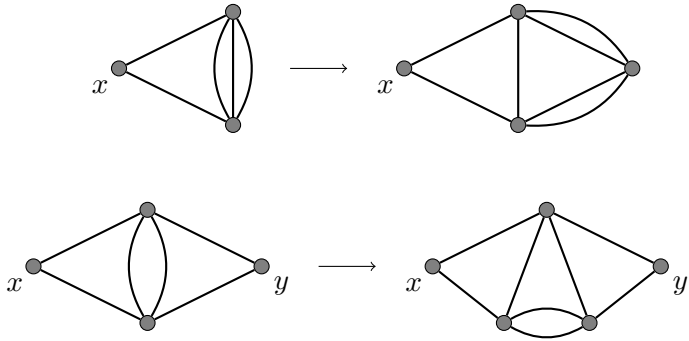
\begin{figure}
\begin{center}
\begin{tikzpicture}[scale=0.75]
\tikzstyle{every node}=[circle, fill=gray, draw=black,inner sep=0.7mm];
\node (x) at (0,0) {};
\node (y) at (0,2) {};
\node[label=below left:$x$] (v) at (-2,1) {};
\draw[style=thick] (x)--(y)--(v)--(x);
\draw[style=thick] (y) to [bend left] (x);
\draw[style=thick] (x) to [bend left] (y);
\draw [->] (1,1)--(2,1);
\pgftransformxshift{5cm}
\node[label=below left:$x$] (v) at (-2,1) {};
\node (x) at (0,0) {};
\node (y) at (0,2) {};
\node (z) at (2,1) {};
\draw[style=thick] (x)--(y)--(z)--(x)--(v)--(y);
\draw[style=thick] (y) to [bend left] (z);
\draw[style=thick] (z) to [bend left] (x);
\pgftransformxshift{-8.5cm}
\pgftransformyshift{-2.5cm}
\node[label=below left:$x$] (x) at (0,0) {};
\node[label=below right:$y$] (y) at (4,0) {};
\node (u) at (2,1) {};
\node (v) at (2,-1) {};
\draw[style=thick] (u)--(x)--(v)--(y)--(u);
\draw[style=thick] (v) to [bend left] (u);
\draw[style=thick] (v) to [bend right] (u);
\draw [->] (5,0)--(6,0);
\pgftransformxshift{7cm}
\node[label=below left:$x$] (x) at (0,0) {};
\node[label=below right:$y$] (y) at (4,0) {};
\node (u) at (2,1) {};
\node (v1) at (1.25,-1) {};
\node (v2) at (2.75,-1) {};
\draw[style=thick] (v2)--(u)--(x)--(v1);
\draw[style=thick] (v2)--(y)--(u)--(v1);
\draw[style=thick] (v2) to [bend left] (v1);
\draw[style=thick] (v1) to [bend left] (v2);
\end{tikzpicture}
\end{center}
\caption{Operations 3 and 4}
\label{fig:ops34}
\end{figure}

The final operation, which is shown in Figure~\ref{fig:op5}, {\em decreases} the number of vertices, and is used only to create triple edges. The
named vertex is the point of attachment of this subgraph to the graph. In this operation, both sides are necessarily induced blocks (that is, maximal $2$-connected subgraphs) of their
respective graphs.

%\begin{figure}
%\begin{center}
%\begin{tikzpicture}[scale=0.75]
%\tikzstyle{every node}=[circle, fill=gray, draw=black,inner sep=0.7mm];
%\node[label=below left:$x$] (v) at (-2,1) {};
%\node (x) at (0,0) {};
%\node (y) at (0,2) {};
%\node (u) at (2,0) {};
%\node (z) at (2,2) {};
%\draw[style=thick] (u)--(x)--(y)--(u)--(z)--(y);
%\draw[style=thick] (z)--(x)--(v)--(y);
%\draw[style=thick] (z) to [bend left] (u);
%\draw [->] (3,1)--(4,1);
%\pgftransformxshift{7cm}
%\node (x) at (0,0) {};
%\node (y) at (0,2) {};
%\node[label=below left:$x$] (v) at (-2,1) {};
%\draw[style=thick] (x)--(y)--(v)--(x);
%\draw[style=thick] (y) to [bend left] (x);
%\draw[style=thick] (x) to [bend left] (y);
%\end{tikzpicture}
%\end{center}
%\caption{Operation 5, which creates a triple edge}
%\label{fig:op5}
%\end{figure}

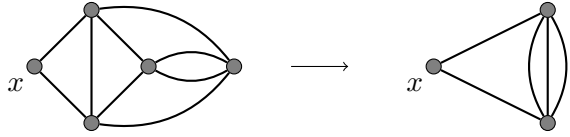
\begin{figure}
\begin{center}
\begin{tikzpicture}[scale=0.75]
\tikzstyle{every node}=[circle, fill=gray, draw=black,inner sep=0.7mm];
\node [label=below left:$x$] (a) at (0,1) {};
\node (b) at (1,0) {};
\node (c) at (1,2) {};
\node (d) at (2,1) {};
\node (e) at (3.5,1) {};
\draw [thick] (a)--(b)--(c)--(a);
\draw [thick] (c)--(d)--(b);
\draw [thick, bend right] (d) to (e);
\draw [thick, bend left] (d) to (e);
\draw [thick, bend left] (c) to (e);
\draw [thick, bend right] (b) to (e);
\draw [->] (4.5,1)--(5.5,1);
%\node (x) at (0.5,0.5) {};
%\node (y) at (0.5,1.5) {};
%\node (z) at (1,1) {};
%\draw [thick] (x)--(y)--(z)--(x);
\pgftransformxshift{9cm}
\node (x) at (0,0) {};
\node (y) at (0,2) {};
\node[label=below left:$x$] (v) at (-2,1) {};
\draw[style=thick] (x)--(y)--(v)--(x);
\draw[style=thick] (y) to [bend left] (x);
\draw[style=thick] (x) to [bend left] (y);
\end{tikzpicture}
\end{center}
\caption{Operation 5, which creates a triple edge}
\label{fig:op5}
\end{figure}

Finally we can state the main theorem of the paper.

\begin{theorem}\label{mainthm}
If $G$ is a connected $4$-regular multigraph with the triangle property, then at least one of the
following holds:
\begin{enumerate}
\item $G$ is the square of a cycle of length at least $3$, or
\item $G$ is the $5$-vertex multigraph shown in Figure~\ref{fig:5ex}, which is obtained by applying Operation $4$ once to the squared $4$-cycle, or 
\item $G$ can be obtained from the line multigraph of a cubic multigraph by a sequence of applications of Operations $1$--$5$.
\end{enumerate}
\end{theorem}

The remainder of the paper proves this theorem.

\section{Proof of the main theorem}

We start with some elementary observations that will be repeatedly used in what follows:

\begin{lemma}\label{noisolated}
A graph $G$ has the triangle property if and only if for every vertex $v \in V(G)$, the graph induced by the neighbourhood 
of $v$ contains no isolated vertices.
\end{lemma}

\begin{proof}
If $w$ is an isolated vertex in $N(v)$, then the edge $vw$ does not lie in a triangle and conversely.
\end{proof}

\begin{lemma}\label{twodegthree}
If $G$ is a connected quartic graph with the triangle property, and $H$ is a subgraph of $G$ such that every vertex of $H$ has degree $4$ other
than two non-adjacent vertices $v$ and $w$ of degree $3$, then $G = H + vw$.
\end{lemma}

\begin{proof}
Suppose that $v$ and $w$ are the two non-adjacent vertices of degree $3$ in $H$. If the fourth edge from $v$ leads to a vertex $x$ {\em outside} $H$, then $x$ is isolated in $N(v)$ (because all the other neighbours of $v$ already have full degree). Thus the fourth edge from $v$ must join $v$ and $w$ and then $H+vw$ is quartic and hence equal to $G$.
\end{proof}

\begin{theorem}
The class of connected quartic multigraphs with the triangle property is closed under Operations 1--5 and their reversals.
\end{theorem}

\begin{proof}
For each of the five operations, it is easy to check that every edge shown in either the left-hand or right-hand subgraph lies in a triangle
completely contained within the subgraph, and so the replacement in either direction does not create any ``bad'' edges not in triangles.  

However, it remains to show that none of the edges that are {\em removed} in the operations or their reversals are essential for creating triangles involving edges that are not shown, either ``optional edges'' with both end vertices inside the subgraph or edges connecting the subgraph to the rest of the graph. For all five operations, any optional edges must connect pairs of named vertices, and it is clear that every pair of named vertices is at distance $2$ in the subgraphs on both sides of each operation (thus forming the necessary triangle if the edge joining them was actually present).

Now consider edges connecting the subgraphs to the rest of the graph. For Operations 1 and 2 (forwards), the requirement that the triangle be eligible ensures that such edges lie in triangles using only edges that will not be removed. For Operations 1 and 2 (backwards) and the remaining operations in either direction, the named vertices are adjacent only to vertices of degree four, and so there can be no triangles
using an edge of the subgraph except those completely contained in the subgraph. 

\end{proof}

The following proposition is the heart of the proof, as it characterises those quartic graphs with the triangle property that have {\em not} arisen
as a consequence of applying Operations $1$--$4$ to a smaller graph.

\begin{propn}\label{mainpart}
Let $G$ be a connected quartic graph with the triangle property on at least $5$ vertices that 
contains none of the subgraphs on the right-hand sides of
Operations $1$--$4$. Then either:
\begin{enumerate}
\item $G$ is the square of a cycle of length at least $7$, or
\item $G$ is obtained from the line multigraph of a cubic multigraph by applications of Operation~$5$.
\end{enumerate}
\end{propn}

\begin{proof}
The proof proceeds via a series of claims progressively restricting the structure of $G$. 

\medskip
\noindent
{\bf Claim 1}: The double edges of $G$ form a matching. 

\medskip
Suppose for a contradiction that $uv$ and $vw$ are both double-edges. By Lemma~\ref{noisolated}, $u$ is adjacent to $w$ and because $G$ has
more than three vertices $u$ is adjacent to a fourth vertex $x$. By Lemma~\ref{noisolated}, the vertex $x$ is forced to be adjacent to $w$, thereby creating
the subgraph on the right-hand side of Operation 3. 

\medskip
\noindent
{\bf Claim 2}: Every double edge is the diagonal of an induced $K_4^-$ (the graph obtained by removing a single edge from $K_4$). 
\medskip

 Let $xy$ be a double edge, and let $v$ be a common neighbor of $x$ and $y$ (which must exist by Lemma~\ref{noisolated}).  Let $x_1$ and $y_1$ be the remaining neighbours of $x$ and $y$, respectively, and note that by Claim 1, $x_1 \ne v \ne y_1$.  If $x_1\ne y_1$, then $x_1v,y_1v\in E(G)$ to create the required triangles, which creates the right-hand graph of Operation $4$ which is a contradiction. Therefore we conclude that $x_1=y_1$ and denote this vertex by $w$.  If $vw \in E(G)$, then $v$ and $w$ must have an additional common neighbour, $z$. Therefore we have deduced the existence of the subgraph shown in Figure~\ref{fig:claim2} which clearly contains $K_{1,1,3}$ (the right-hand graph of Operation 2) which is again a contradiction.  Therefore $vw \notin E(G)$ and so $\{x,y,v,w\}$ form an induced $K_4^-$.

\begin{figure}

\begin{center}
\begin{tikzpicture}[bend angle = 15]
\tikzstyle{every node}=[circle, fill=gray, draw=black,inner sep=0.7mm];
\node[label=below left:$x$] (x) at (0,0) {};
\node[label=above left:$y$] (y) at (0,2) {};
\node[label=right:$z$] (z) at (2,1) {};
\node [label = above:$v$] (u) at (1,1.5) {};
\node  [label = below:$w$](v) at (1,0.5) {};
\draw[style=thick] (v)--(x)--(u)--(y)--(v)--(z)--(u)--(v);
\draw [style=thick, bend right] (x) to (y);
\draw [style=thick, bend left] (x) to (y);
\end{tikzpicture}
\end{center}
\caption{The right-hand graph of Operation 2 appears in Claim 2}
\label{fig:claim2}
\end{figure}
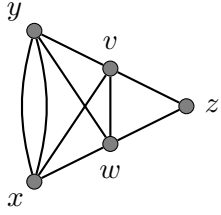

\medskip
\noindent
{\bf Claim 3}: If $G$ contains an induced $K_4^-$ with no multiple edges, then $G$ is a squared cycle of length at least $7$.

\medskip
%
% If all induced $K_4^-$ contain double edges, this implies that $G$ is the line multigraph of a $3$-regular multigraph, in which every triangle contains a double edge, and we are done. This can be seen as follows. From every double edge, delete one of the two vertices. The resulting graph has no induced $K_{1,3}$ and no induced $K_4^-$, and is therefore the line graph of a triangle free graph $H$. Doubling up the edges in $H$ representing the partners of the deleted vertices creates the line multigraph original of $G$.
% 
% So we may assume that $G$ is in fact a simple graph. If $G$ contains no induced $K_4^-$, then $G$ is the line graph of a $3$-regular triangle-free graph, and again we are done. So we may assume that $G$ contains an induced $K_4^-$ 
Suppose that $G$ contains an induced $K_4^-$ on the vertices $\{v_1,v_2,v_3,v_4\}$ with all edges $v_iv_j\in E(G)$ except $v_1v_2$. 
Let $w_3,w_4$ be the remaining neighbors of $v_3$ and $v_4$, respectively. If $w_3=w_4$, then $G[\{v_1,v_2,v_3,v_4,w_3\}]$ contains a $K_{1,1,3}$, which is the graph on the right side of Operation~2, a contradiction. Thus we can assume that $w_3 \ne w_4$ obtaining
the first graph of Figure~\ref{fig:claim3}.
 
Now, to avoid $w_3$ being isolated in the neighbourhood of $v_3$, it must be adjacent to either $v_1$ or $v_2$, and similarly for $w_4$. However if either $v_1$ or $v_2$ is adjacent to {\em both of} $w_3$ and $w_4$, then this creates the graph on the right-hand side of Operation 1. 
Therefore, by symmetry we can assume that $v_1w_3,v_2w_4\in E(G)$ and $v_1w_4,v_2w_3\not\in E(G)$.  Continuing, we see that the vertices $v_1$ and $v_2$ must each have a fourth neighbour $w_1$ and $w_2$, respectively. To avoid $w_1$ being isolated in $N(v_1)$ we must have $w_1w_3 \in E(G)$ and to avoid $w_2$ being isolated in $N(v_2)$ we must have $w_2w_4 \in E(G)$. If $w_1 = w_2$ then we have the
situation of Lemma~\ref{twodegthree} and so $G$ is the square of the $7$-cycle. 

If $w_1 \ne w_2$, then we have arrived at the second
graph in Figure~\ref{fig:claim3} and we can continue this process.
Note that $w_3w_4\notin E(G)$ as $w_3$ and $w_4$ have no common neighbor amongst the previously named vertices.
So consider neighbours $x_3$ and $x_4$ of $w_3$ and $w_4$, respectively.  If $x_3=w_2$, then the final neighbour of $w_2$ is a neighbor of $w_3$ (to create a triangle for $w_2w_3$) and a neighbor of $w_4$ (to create a triangle for $w_4x_4$), and thus $x_4=w_1$, and $w_1w_2\in E(G)$. In this case, $G$  is the square of an $8$-cycle.

So assume that $x_3\ne w_2$ and $x_4\ne w_1$. If $x_3=x_4$, this forces $w_1w_2\in E(G)$, and $G$ is the square of a $9$-cycle. If $x_3\ne x_4$, we continue with vertices $x_1$ and $x_2$, and so on. At each stage of this process, there is a chain of $K_4^-$s and we consider the two missing neighbours of the two vertices of degree three. Either the two neighbours are both in the chain already, in which case $G$ is an even
squared cycle, or the two neighbours coincide in a single new vertex, in which case $G$ is an odd squared cycle, or they are two new vertices, in which case the chain is extended, and the argument repeated. Eventually this process must stop, producing a squared cycle.

\begin{figure}
\begin{center}
\begin{tikzpicture}[scale=0.7]
\tikzstyle{every node}=[circle, fill=gray, draw=black,inner sep=0.7mm];
\node [label = below:{\small $v_1$}] (v1) at (2,0) {};
\node [label = above:{\small $v_2$}] (v2) at (0,2) {};
\node [label = above:{\small $v_3$}] (v3) at (2,2) {};
\node [label = below:{\small $v_4$}] (v4) at (0,0) {};
\node [label = above:{\small $w_3$}] (w3) at (4,2) {};
\node [label = below:{\small $w_4$}] (w4) at (-2,0) {};
\draw [thick] (v1)--(v3)--(v2)--(v4)--(v1);
\draw [thick] (v4)--(w4);
\draw [thick] (v3)--(w3);
\draw [thick] (v3)--(v4);
\pgftransformxshift{9cm}
\node [label = below:{\small $v_1$}] (v1) at (2,0) {};
\node [label = above:{\small $v_2$}] (v2) at (0,2) {};
\node [label = above:{\small $v_3$}] (v3) at (2,2) {};
\node [label = below:{\small $v_4$}] (v4) at (0,0) {};
\node [label = above:{\small $w_3$}] (w3) at (4,2) {};
\node [label = below:{\small $w_4$}] (w4) at (-2,0) {};
\node [label = above:{\small $w_2$}] (w2) at (-2,2) {};
\node [label = below:{\small $w_1$}] (w1) at (4,0) {};
\draw [thick] (v1)--(v3)--(v2)--(v4)--(v1);
\draw [thick] (v4)--(w4);
\draw [thick] (v3)--(w3);
\draw [thick] (v3)--(v4);
\draw [thick] (v2)--(w4);
\draw [thick] (v1)--(w3);
\draw [thick] (v2)--(w2);
\draw [thick] (v1)--(w1);
\draw [thick] (w2)--(w4);
\draw [thick] (w1)--(w3);
\end{tikzpicture}
\end{center}
\caption{Two stages in the construction of Claim 3}
\label{fig:claim3}
\end{figure}
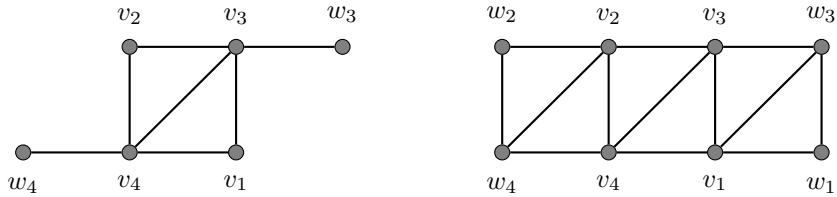

\medskip
\noindent
{\bf Claim 4}: If $G$ does not contain an induced $K_4^-$ with no multiple edges, then $G$ can be obtained from the line multigraph of a cubic multigraph by applications of Operation 6.

\medskip
Create a simple graph $S$ from $G$ as follows. For every double edge, delete one of the two end vertices; as these vertices have identical neighbourhoods (i.e., they are {\em twins}) it does not matter which one. For every triple edge, delete two of the three edges. This graph does not contain $K_{1,3}$ and $K_4^-$ as induced subgraphs, and so by Harary \& Holzmann \cite{HararyHolzmann} it is the linegraph of a unique triangle free graph $L^{-1}(F)$. Now $G$ can be reconstructed from $L^{-1}(F)$ as follows: double every edge in $L^{-1}(F)$ that corresponds to a vertex in $G$ whose twin was deleted. Add a double edge between the vertices of degree $1$ on the edges corresponding to the end vertices of each triple edge in $G$. This forms a cubic multigraph, and we can construct $G$ by taking the line multigraph of this graph and performing Operation~$5$ to recover the triple edges. 
\end{proof}

With these results, we are now in a position to prove the main theorem.

\begin{proof} (of Theorem~\ref{mainthm})
Suppose that $G$ is a quartic graph with the triangle property, and repeatedly perform the reverse of Operations 1--4 until the resulting graph $G'$ has no subgraphs isomorphic to any of the graphs on the right-hand side of Operations 1--4. If $G'$ has at least $5$ vertices, then by Proposition~\ref{mainpart}, it is either a squared $n$-cycle for $n \geq 7$ or has been obtained from the line multigraph of a cubic multigraph by applications of Operation 5. In the former case, $G$ itself is equal to the squared $n$-cycle, because for $n \geq 7$, the squared $n$-cycles contain no eligible triangles. In the latter case, combining the applications of Operation 5 that
transform the line multigraph of a cubic multigraph into $G'$ with the applications of Operations 1--4 that transform $G'$ into $G$ shows that 
$G$ has the required structure. If $G'$ has fewer than $5$ vertices then it is either the squared $3$-cycle or the squared $4$-cycle. The squared $3$-cycle is the linegraph of a triple edge, while the only operation that can be applied to the squared $4$-cycle is Operation 4 which creates the graph of Figure~\ref{fig:5ex} (to which no further operations can be applied.) \end{proof}

{\bf Remark:} Some graphs appear in more than one of the classes of Theorem~\ref{mainthm}. In particular, the squared $3$-cycle is also the line graph of a triple edge, the squared $5$-cycle (that is, $K_5$) is obtained by applying Operation 2 to the squared $3$-cycle, while the squared $6$-cycle is obtained by applying Operation 1 to the squared $3$-cycle.

\bigskip

With some more work, we can get a more precise result characterising the {\em simple} quartic graphs with the triangle property. First, a preliminary lemma; this lemma refers to the two graphs in Figure~\ref{fig:blocks} which are each obtained from the graph on the left-hand side of Operation
5 (Figure~\ref{fig:op5}) by performing Operation 1 on one of the two inequivalent  choices of eligible triangle.

\begin{figure}
\begin{tikzpicture}[scale=1]
\tikzstyle{every node}=[circle, fill=gray, draw=black,inner sep=0.7mm];
\node (a) at (0,1) {};
\node (b) at (1,0) {};
\node (c) at (1,2) {};
\node (d) at (2,1) {};
\node (e) at (3.5,1) {};
\draw [thick] (a)--(b)--(c)--(a);
\draw [thick] (c)--(d)-- node (z) {} (b);
\draw [thick, bend right] (d) to node (x) {} (e);
\draw [thick, bend left] (d) to (e);
\draw [thick, bend left] (c) to (e);
\draw [thick, bend right] (b) to node (y) {} (e);
\draw [thick] (x)--(y)--(z)--(x);
\pgftransformxshift{5cm}
\node (a) at (0,1) {};
\node (b) at (1,0) {};
\node (c) at (1,2) {};
\node (d) at (2,1) {};
\node (e) at (3.5,1) {};
\draw [thick] (a)--(b)--(c)--(a);
\draw [thick] (c)--(d)--(b);
\draw [thick, bend right] (d) to (e);
\draw [thick, bend left] (d) to (e);
\draw [thick, bend left] (c) to (e);
\draw [thick, bend right] (b) to (e);
\node (x) at (0.5,0.5) {};
\node (y) at (0.5,1.5) {};
\node (z) at (1,1) {};
\draw [thick] (x)--(y)--(z)--(x);
\end{tikzpicture}
\caption{Applying Operation 1 to the eligible triangles in the left-hand graph of Operation 5}
\label{fig:blocks}
\end{figure}
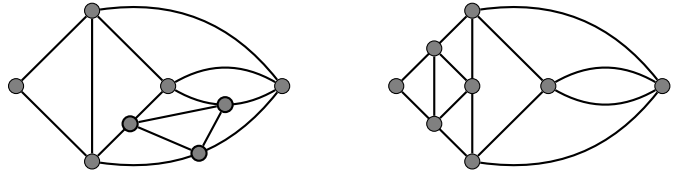

\begin{lemma}
Suppose that $G$ is a quartic graph with the triangle property and that $B$ is a block of $G$ that contains no $K_{1,1,3}$. Applying Operation 1
in reverse to a configuration contained in $B$ creates a $K_{1,1,3}$ in $B$ if and only if $B$ is one of the two blocks shown in 
Figure~\ref{fig:blocks}.
\end{lemma}

\begin{proof}
The ``if'' direction is obvious, so we prove only the ``only if''. So, suppose 
that the new triangle $T$ created by performing Operation 1 in reverse creates (part of) a subgraph 
$K$ isomorphic to  $K_{1,1,3}$. 
First we note that $|T \cap E(K)| \not= 1$ because every edge of $K$ is incident to a vertex of degree $4$. Therefore,  $|T \cap E(K)| \in \{2,3\}$ and, without loss of generality, $B$ contains one of the configurations shown in Figure~\ref{fig:op1rev}.

Consider first the case when $|T \cap E(K)| = 2$, in which case $B$ contains the first configuration shown in Figure~\ref{fig:op1rev}. Now the fourth
neighbour of $y$ must be adjacent to $v$ and so is either $x$ or $z$. If $xy \in E(G)$, then by Lemma~\ref{twodegthree}, $xz \in E(G)$ and $G$ is an $8$-vertex graph, which contains a $K_{1,1,3}$ with $xv$ as the edge joining the two vertices of degree $4$, contradicting the hypothesis.
 So $yz \in E(G)$, which yields the first configuration shown in Figure~\ref{fig:blocks}.

Next consider the case when $|T \cap E(K)| = 3$, in which case $B$ contains the second configuration shown in Figure~\ref{fig:op1rev}. By Lemma~\ref{noisolated}, it follows that $xy \in E(G)$. If $xz \in E(G)$, then $yz \in E(G)$ which creates a $K_{1,1,3}$ with $xy$ as the edge joining the two vertices of degree four, contradicting the hypothesis. Otherwise, either $x$ and $y$ have a new mutual neighbour, contradicting the
fact that $B$ is $K_{1,1,3}$-free, or $xy$ is a double-edge, in which case we have the second configuration
shown in Figure~\ref{fig:blocks}.
\end{proof}

\begin{figure}
\begin{tikzpicture}[scale=1.4]
\tikzstyle{every node}=[circle, fill=gray, draw=black,inner sep=0.7mm];
\node[label=below left:$x$] (x) at (0,0) {};
\node[label=above left:$y$] (y) at (0,2) {};
\node[label=right:$z$] (z) at (2,1) {};
\node[label=above:$u$] (u) at (1,1.4) {};
\node [label=below:$v$] (v) at (1,0.6) {};
\draw[style=thick] (v)--(x)--(u)--(y)--(v)--(z)--(u)--(v);
\node (a) at (0.5,1.7) {};
\node (b) at (1.5,1.2) {};
\node (c) at (1.4,2.3) {};
\draw[style=thick, bend left, out=60,in=120](a) to (b);
\draw[style = thick, bend right] (b) to (c);
\draw[style = thick, bend right] (c) to (a);
\draw [style=thick, bend left] (y) to (c);
\draw [style=thick, bend left] (c) to (z);
\pgftransformxshift{4cm}
\pgftransformyshift{0.1cm}
\node[label=below left:$x$] (x) at (0,0) {};
\node[label=above left:$y$] (y) at (0,2) {};
\node[label=right:$z$] (z) at (2,1) {};
\node [label=above:$u$] (u) at (1,1.6) {};
\node [label=below:$v$] (v) at (1,0.4) {};
\draw[style=thick] (v)--(x)--(u)--(y)--(v)--(z)--(u)--(v);
\node (a) at (1.5,1.3) {};
\node (b) at (1.5,0.7) {};
\node (c) at (1,1) {};
\draw[style=thick] (a)--(b)--(c)--(a);
\end{tikzpicture}
\caption{Reversing Operation 1 to create a $K_{1,1,3}$ either creates the edges $\{yu, uz\}$ or the entire triangle $\{uv, vz, zu\}$.}
\label{fig:op1rev}
\end{figure}
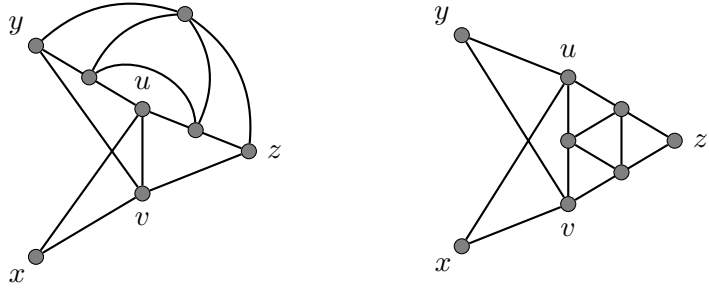

\begin{cor}
If $G$ is a simple connected quartic graph with the triangle property on at least $5$ vertices, then either:
\begin{enumerate}
\item $G$ is the square of a cycle of length at least $5$, or
\item $G$ is obtained from the line multigraph of a cubic multigraph by repeatedly replacing triangles by copies of $K_{1,1,3}$, 
ensuring that all multiple edges are eliminated.
\end{enumerate}
\end{cor}

\begin{proof}

Suppose that $G$ is not the square of a cycle of length $7$ or more. Let $G_1$ be a $K_{1,1,3}$-free graph obtained from $G$ by 
performing the reverse of Operation 2 until it can no longer be performed, and then let $G_2$ be a graph obtained from $G_1$ by performing the reverse of Operation 1 as many times as possible. By the previous lemma, $G_2$ may now have some $K_{1,1,3}$-subgraphs, but
only in blocks isomorphic to the left-hand graph of Operation 5. Finally, let $G_3$ be the graph obtained from $G_2$ by applying Operation 5 to each of these blocks (thereby creating some triple edges).  Our aim is to show that $G_1$ is actually the line multigraph graph of a cubic multigraph.

Now $G_3$ has no subgraphs isomorphic to the right-hand graphs of Operations 1--4, and 
so it satisfies the hypotheses of Proposition~\ref{mainpart}. As it cannot be the square of a cycle of length at least $7$, it is therefore equal to the 
line multigraph $L(M)$ of a cubic multigraph $M$, possibly with some applications of Operation 5 creating all the blocks consisting of a triangle with a triple edge.  However, as all of these blocks in $G_3$ originally arose by applying Operation 5 to the relevant blocks of $G_2$, it follows that $G_2$ itself is isomorphic to $L(M)$. 

Finally, we note that performing Operation 1 on the line multigraph $L(M)$ of a cubic multigraph is the same as replacing a vertex of $M$ by a triangle, and {\em then} taking the line multigraph, and so $G_1$ is itself the line multigraph of a cubic multigraph.  Finally, $G$ is obtained from $G_1$ by repeated application of Operation 2, which is the replacing of triangles with copies of $K_{1,1,3}$.  As $G$ is simple, this replacement must eliminate all multiple edges.
\end{proof}

{\bf Remarks:} 

\smallskip

\begin{enumerate}
\item From the main theorem, we can immediately see that $G$ is obtained from the line multigraph of a cubic multigraph by the application of 
{\em some} sequence of Operations 1 and 2. The point of this corollary is to show that all the applications of Operation 2 can be moved to the {\em end} of the sequence, and all the applications of Operation 1 can be subsumed into the selection of the cubic multigraph. The key idea of the proof
is very simple, with all the awkward detail just dealing with the blocks of Figure~\ref{fig:blocks}.

\item It is {\em not} possible to construct all the simple quartic graphs with the triangle property starting only with the line graphs of {\em simple} cubic graphs, and so venturing into multigraphs is necessary even to get the result for simple graphs.
\end{enumerate}

\section{Conclusion}

It is natural to ask whether a similar result can be obtained for $5$-regular graphs --- in fact, determining the $5$-regular graphs with
the triangle property arose in the context of a different problem, and was the problem that originally motivated this research. However, computer investigations reveal that the number of quintic graphs with the triangle property grows rather rapidly, even when 
restricted to simple graphs, and it seems that a result similar to the one for quartic graphs in this paper would be extremely complicated.

For example, Figure~\ref{fig:family1} and Figure~\ref{fig:family2} give examples of $5$-regular graphs with the triangle property each of which can be extended in an obvious way to form an infinite family.

\begin{figure}
\begin{tikzpicture}[scale=0.8]
\tikzstyle{every node}=[circle, fill=gray, draw=black,inner sep=0.7mm];
\node (v0) at (0:1cm) {};
\node (v1) at (60:1cm) {};
\node (v2) at (120:1cm) {};
\node (v3) at (180:1cm) {};
\node (v4) at (240:1cm) {};
\node (v5) at (300:1cm) {};
\node (w0) at (0:2cm) {};
\node (w1) at (60:2cm) {};
\node (w2) at (120:2cm) {};
\node (w3) at (180:2cm) {};
\node (w4) at (240:2cm) {};
\node (w5) at (300:2cm) {};
\draw [thick] (v0)--(v1)--(v2)--(v3)--(v4)--(v5)--(v0);
\draw [thick] (w0)--(w1)--(w2)--(w3)--(w4)--(w5)--(w0);
\draw [thick] (w0)--(v1)--(w2)--(v3)--(w4)--(v5)--(w0);
\draw [thick] (v0)--(w1)--(v2)--(w3)--(v4)--(w5)--(v0);
\draw [thick] (v0)--(w0);
\draw [thick] (v1)--(w1);
\draw [thick] (v2)--(w2);
\draw [thick] (v3)--(w3);
\draw [thick] (v4)--(w4);
\draw [thick] (v5)--(w5);
\end{tikzpicture}
\caption{The lexicographic product $C_n[K_2]$, shown here for $n=6$}
\label{fig:family1}
\end{figure}
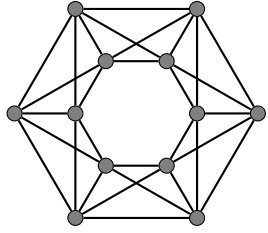

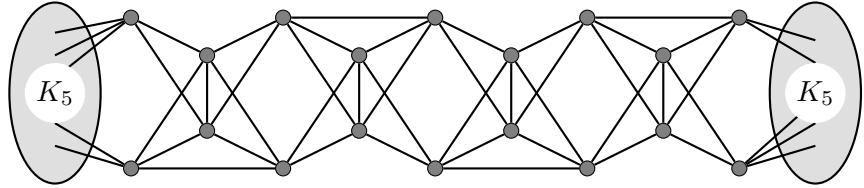
\begin{figure}
\begin{tikzpicture}
\tikzstyle{every node}=[circle, fill=gray, draw=black,inner sep=0.7mm];
\node (v0) at (0,2) {};
\node (v1) at (0,0) {};
\node (v2) at (1,1.5) {};
\node (v3) at (1,0.5) {};
\node (v4) at (2,2) {};
\node (v8) at (4,2) {};
\node (v12) at (6,2) {};
\node (v16) at (8,2) {};
\node (v5) at (2,0) {};
\node (v9) at (4,0) {};
\node (v13) at (6,0) {};
\node (v17) at (8,0) {};
\node (v6) at (3,0.5) {};
\node (v7) at (3,1.5) {};
\node (v10) at (5,1.5) {};
\node (v11) at (5,0.5) {};
\node (v14) at (7,1.5) {};
\node (v15) at (7,0.5) {};

\draw [thick] (v0)--(v2);
\draw [thick] (v0)--(v3);
\draw [thick] (v1)--(v2);
\draw [thick] (v1)--(v3);
\draw [thick] (v1)--(v5);
\draw [thick] (v2)--(v3);
\draw [thick] (v2)--(v4);
\draw [thick] (v2)--(v5);
\draw [thick] (v3)--(v4);
\draw [thick] (v3)--(v5);
\draw [thick] (v4)--(v6);
\draw [thick] (v4)--(v7);
\draw [thick] (v4)--(v8);
\draw [thick] (v5)--(v6);
\draw [thick] (v5)--(v7);
\draw [thick] (v6)--(v7);
\draw [thick] (v6)--(v8);
\draw [thick] (v6)--(v9);
\draw [thick] (v7)--(v8);
\draw [thick] (v7)--(v9);
\draw [thick] (v8)--(v10);
\draw [thick] (v8)--(v11);
\draw [thick] (v9)--(v10);
\draw [thick] (v9)--(v11);
\draw [thick] (v9)--(v13);
\draw [thick] (v10)--(v11);
\draw [thick] (v10)--(v12);
\draw [thick] (v10)--(v13);
\draw [thick] (v11)--(v12);
\draw [thick] (v11)--(v13);
\draw [thick] (v12)--(v14);
\draw [thick] (v12)--(v15);
\draw [thick] (v12)--(v16);
\draw [thick] (v13)--(v14);
\draw [thick] (v13)--(v15);
\draw [thick] (v14)--(v15);
\draw [thick] (v14)--(v16);
\draw [thick] (v14)--(v17);
\draw [thick] (v15)--(v16);
\draw [thick] (v15)--(v17);
\draw [thick,fill=gray!25!white] (-1,1) ellipse (0.6cm and 1.2cm);
\draw [thick] (-1,0.3) -- (v1);
\draw [thick] (-1,0.6) --(v1);
\draw[thick]  (-1,1.2) -- (v0);
\draw [thick] (-1,1.8) --(v0);
\draw[thick]  (-1,1.5) --(v0);
\node [fill=white,draw=white] at (-1,1) {$K_5$};
\draw [thick,fill=gray!25!white] (9,1) ellipse (0.6cm and 1.2cm);
\draw[thick]  (9,0.3) -- (v17);
\draw[thick]  (9,0.6) -- (v17);
\draw [thick] (9,0.9) -- (v17);
\draw[thick]  (9,1.7) --(v16);
\draw [thick] (9,1.4) --(v16);
\node [fill=white,draw=white] at (9,1) {$K_5$};

\end{tikzpicture}
\caption{An early member of a family of $5$-regular graphs with the triangle property}
\label{fig:family2}
\end{figure}

\end{document}